\documentclass{article}
\usepackage[utf8]{inputenc}
\usepackage[english]{babel}
\usepackage{geometry}
\usepackage{pdfpages}
\usepackage{url}
\usepackage[normalem]{ulem}

\usepackage{amssymb}
\usepackage{amsthm}
\usepackage{amsmath}
\usepackage{setspace}
\usepackage{listings}
\usepackage[normalem]{ulem}

\usepackage{algorithmicx}
\usepackage[noend]{algpseudocode}
\usepackage{varwidth}

\newtheorem{theorem}{Theorem}

\newtheorem{lemma}[theorem]{Lemma}
\newtheorem{corollary}[theorem]{Corollary}

\theoremstyle{definition}
\newtheorem{definition}[theorem]{Definition}

\newcommand{\il}[1]{\text{\textbf{IL}\textsf{#1}}}
\newcommand{\ilgen}[1]{\text{\textbf{IL}\textsubscript{gen}\textsf{#1}}}
\newcommand{\kgen}[1]{\text{(\textsf{{#1}})\textsubscript{gen}}}
\newcommand{\R}{\ensuremath{{\widetilde{R}}}}

\newcommand{\ilmn}[0]{\il{M\textsubscript{0}}}

\newcommand{\Rt}[0]{ \widetilde{R} }		
\newcommand{\St}[1]{ \widetilde{S}_{[#1]} }		
\newcommand{\Swt}[0]{ \St{w} }

\usepackage{amsmath}
\usepackage{tikz}
\usepackage{mathdots}
\usepackage{yhmath}
\usepackage{cancel}
\usepackage{color}
\usepackage{siunitx}
\usepackage{array}
\usepackage{multirow}
\usepackage{amssymb}
\usepackage{gensymb}
\usepackage{tabularx}
\usepackage{booktabs}
\usepackage{wasysym}
\usetikzlibrary{fadings}

\usepackage{datetime}
\title{Interpretability logics and generalized Veltman semantics}

\author{Luka~Mikec\footnote{Department of Mathematics, Faculty of Science, University of 
Zagreb, Croatia\newline \hspace*{5.2mm}Supported by Croatian Science Foundation (HRZZ) under 
the projects UIP--05--2017--9219 and IP--01--2018--7459.} \\
\texttt{luka.mikec@math.hr}
\and
Mladen~Vukovi\'{c}\footnote{Department of Mathematics, Faculty of Science, University of Zagreb, 
Croatia\newline \hspace*{5.2mm}Supported by Croatian Science Foundation (HRZZ) under 
the project IP--01--2018--7459}\\
\texttt{vukovic@math.hr}}

\begin{document}

\date{ \today, \currenttime }
\maketitle

\begin{abstract}
We obtain modal completeness of the interpretability logics 
\il{P\textsubscript{0}} and \il{R}
w.r.t.\ generalized Veltman semantics.
Our proofs are based on the notion of full labels \cite{Bilkova-Goris-Joosten}.
We also give shorter proofs of completeness w.r.t.\ generalized semantics for many classical 
interpretability logics.
We obtain decidability and finite model property w.r.t.\ generalized semantics for 
\il{P\textsubscript{0}} and \il{R}.
Finally, we develop a construction that might be useful for proofs of completeness of
extensions of \il{W} w.r.t. generalized semantics in the future, and demonstrate its usage with 
$\il{W*} = \il{M\textsubscript{0}W}$.
\end{abstract}


\section{Introduction}
\subsection{Interpretability logics}

The language of interpretability logics is given by $A::=p\,|\,\bot\,|\,A\to A\,|\,A\rhd A,$
where $p$ ranges over a countable set of propositional variables. 
Other Boolean connectives are defined as abbreviations, as usual. 
Since $\square A$ can be defined (over extensions of \il) as an abbreviation too (exapnded to 
$\neg A\rhd\bot$), we do not include $\square$ or $\Diamond$ in the language.
If $A$ is constructed in this way, we will say that $A$ is a modal formula.

\begin{definition}
	Interpretability logic \il{} is given by the following list of axiom schemata.
	\begin{enumerate}
		\item classical tautologies (in the new language);
		\item[K] $\square (A\rightarrow B)\rightarrow (\square A\rightarrow \square B)$;
		\item[L] $\square (\square A\rightarrow A)\rightarrow \square A$;
		\item[J1] $\square (A\rightarrow B)\rightarrow A\rhd B$;
		\item[J2] $(A\rhd B) \wedge (B\rhd C) \  \rightarrow \  A\rhd C$;
		\item[J3] $(A\rhd C) \wedge (B\rhd C) \  \rightarrow \  A\vee B\rhd C$;
		\item[J4] $A\rhd B\rightarrow(\Diamond A \rightarrow \Diamond B)$;
		\item[J5] $\Diamond A \rhd A$.
	\end{enumerate}
	Rules of inference are modus ponens and necessitation (generalization).
\end{definition}

We treat $\rhd$ as having higher priority than $\to$, but lower than other logical connectives. 
Other interpretability logics are obtained by extending \il{} with further schemata (``principles of interpretability'').

Let us briefly describe the motivation behind studying interpretability logics. 
In a sufficiently strong formal theory $T$ in the language $\mathcal L_T$, one can construct a binary interpretability predicate $\mathsf{Int}_T$. 
This predicate expresses that one finite extension of $T$ interprets another finite extension of $T$.
Any mapping $A \mapsto A^*$ where $A$ is a modal formula and $A^* \in \mathcal L_T$, such that:
\begin{itemize}
	\item  it commutes with logical connectives;
	\item  if $p$ is a propositional variable, $p^*$ is a sentence;
	\item  $(A \rhd B)^*$ = $\mathsf{Int}_T(\lceil A^* \rceil, \lceil B^* \rceil)$, where $\lceil X \rceil$ is the numeral of the Gödel number of $X$;
\end{itemize}
is called \textit{an arithmetical interpretation}.
The \textit{interpretability logic of} $T$, denoted by $\il{(T)}$, is the set of all modal formulas $A$
such that $T \vdash A^*$ for all arithmetical interpretations.
Berarducci \cite{Berarducci}  and Shavrukov \cite{Shavrukov-88} independently proved that $\il{(T)}=\il{M}$, if 
$T$ is an essentially  reflexive theory.
Visser \cite{Visser90} proved that $\il{(T)}=\il{P}$, if $T$ is finitely axiomatizable and 
contains \texttt{supexp}.

One of the major open problems in the field is to determine \il{(All)}, the interpretability logic of all 
``reasonable arithmetical theories.''
The ongoing search for \il{(All)} is the main motivation behind studying extensions of \il{}. Studying modal properties of lower bounds of \il{(All)} turns out to be useful for finding new principles within \il{(All)}. For example, the principle \textsf{R} was discovered while trying to prove modal completeness of \textsf{ILP\textsubscript{0}W} \cite{Goris-Joosten-11}. 

Definitions and furher details can be found in e.g.\ \cite{Visser98}.

\subsection{Semantics}
The most commonly used semantics for the interpretability logic \il{} and its extensions is the \textit{Veltman semantics} (or \textit{ordinary Veltman semantics}). 
\begin{definition}[\cite{deJongh-Veltman-90}, Definition 1.2]
	A Veltman frame $\mathfrak{F}$ is a structure $(W,R, \{S_w : w\in W \})$, 
	where $W$ is a non-empty set, $R$ is a transitive and converse well-founded binary relation on 
	$W$ and for all $w\in W$ we have:
		\begin{itemize}
			\item[a)] $S_w\subseteq R[w]^2$, where $R[w]=\{x \in W : wRx\}$; 
			\item[b)] $S_w$ is  reflexive on $R[w]$; 
            \item[c)] $S_w$ is transitive;
			\item[d)] if $wRuRv$ then $uS_w v$.
		\end{itemize}
\end{definition}
A \textit{Veltman model} is a quadruple $\mathfrak{M}=(W,R, \{S_w : w\in W \}, \Vdash)$,
    where the first three components form a generalized Veltman frame.
The forcing relation $\Vdash$ is extended as usual in Boolean cases, 
and $w\Vdash A\rhd B$ holds if and only if for all $u$ such that $wRu$ and $u\Vdash A$ 
there exists $v$ such that $uS_w v$ and $v\Vdash B$.

In what follows we will mainly use the following semantics, which we will refer to as the \textit{generalized Veltman semantics}. De Jongh defined this specific generalization of Veltman semantics. The main purpose of its introduction, and until recently the only usage, was to show independence of certain extensions of \il{}, by Verbrugge (private correspondence), Vukovi\'c \cite{Vukovic99} and Goris and Joosten \cite{Goris-Joosten-11}.

\begin{definition}
	A generalized Veltman frame $\mathfrak F$ is a structure $(W,R,\{S_w:w\in W\})$, 
	where $W$ is a non-empty set, $R$ is a transitive and converse well-founded binary relation on 
$W$ and for all $w\in W$ we have:
	\begin{itemize}
		\item[a)] $S_w\subseteq R[w]\times \left(\mathcal{P}( {R[w]})\setminus\{\emptyset\}\right)$;
		\item[b)] $S_w$ is quasi-reflexive: $wRu$ implies $uS_w\{u\}$;
		\item[c)] $S_w$ is quasi-transitive: if $uS_wV$ and $vS_wZ_v$ for all $v\in V$, then  $uS_w(\bigcup_{v\in V}Z_v)$;
		\item[d)] if $wRuRv$, then $uS_w\{v\}$;
		\item[e)] monotonicity: if $uS_wV$ and $V\subseteq Z\subseteq R[w]$, then $uS_wZ$.
	\end{itemize}
\end{definition}
A \textit{generalized Veltman model} is a quadruple $\mathfrak{M}=(W,R, \{S_w : w\in W \}, \Vdash)$,
    where the first three components form a Veltman frame. 
Now $w\Vdash A\rhd B$ holds if and only if for all $u$ such that $wRu$ and $u\Vdash A$ 
there exists $V$ such that $uS_w V$ and $V\Vdash B$. 
We write $V \Vdash B$ if $v \Vdash B$ for all $v \in V$.

\subsection{Principles, completeness and decidability}
Let us review some relevant results and approaches.  When we need to refer to an extension of \il{} (an arbitrary extension if not stated otherwise), we will write \il{X}. 

Let (\textsf{X}) (resp.\ \kgen{X}) denote a formula of first-order or higher-order logic  such that for all ordinary (resp.\ generalized) Veltman 
frames $\mathfrak{F}$
the following holds:
$$ \mathfrak{F}\Vdash \mathsf{X} \ \mbox{ if and only if } \ 
    \mathfrak{F}\models (\mathsf{X}) \text{\ (resp.\ } \mathfrak{F}\models \kgen{X} \text{)}. $$
Formulas $(\mathsf{X})$ and \kgen{X} are called characteristic properties (or frame conditions) of the given logic \il{X}. 
The class of all ordinary (resp.\ generalized) Veltman frames $\mathfrak{F}$ such that $\mathfrak{F}\models (\mathsf{X}) \text{\ (resp.\ } \mathfrak{F}\models \kgen{X} \text{)}$ is called the chararacteristic class of (resp.\ generalized) frames for \il{X}.
If $\mathfrak{F}\models \kgen X$ we also say that the frame $\mathfrak F$ possesses the property \kgen{X}.  
We say that an ordinary (resp.\ generalized) Veltman model $\mathfrak{M}=(W,R,\{ S_w:w\in W\},\Vdash)$ is an 
\il{X}-model (resp.\ \ilgen{X}-model), or that model $\mathfrak{M}$ possesses the property (\textsf{X}) (resp.\ 
\kgen{X}), if the frame $(W,R,\{S_w:w\in W\})$ possesses the property (\textsf{X}) (resp. \kgen{X}). 
A logic \il{X} will be said to be complete with respect to ordinary (resp.\ generalized) semantics if for all 
modal formulas $A$ we have that validity of $A$ over all \il{X}-frames (resp.\ all \ilgen{X}-frames) implies 
$\il{X} \vdash A$.

We say that \il{X} has finite model property (FMP) w.r.t.\ ordinary (resp.\ generalized) semantics if for each 
formula $A$ satisfiable in some \il{X}-model (resp.\ \ilgen{X}-model), $A$ is also satisfiable in some finite 
\il{X}-model (resp.\ \ilgen{X}-model).

If we include results from the current paper, we have the
following table. Here, \textit{o} stands for the ordinary Veltman semantics, and \textit{g} for the generalized 
Veltman semantics (as defined earlier).\footnote{If we adhere to our definitions, \il{P\textsubscript{0}} has FMP w.r.t.\ ordinary semantics, but this is not a useful result since it is incomplete w.r.t.\ ordinary semantics. }

\begin{center}
\begin{tabular}{c|c|c|c|c|c}
& principle & compl.\ (o) & compl.\ (g) & FMP (o) &  FMP (g)  \\
\hline
\textsf{M} & $A\rhd B\to A\wedge\square C \rhd B\wedge\square C$ & +  & + & + & + \\
\hline
\textsf{M\textsubscript{0}} & $A\rhd B\to\Diamond A\wedge\square C\rhd B\wedge\square C$ & +  & + & ? & +\\
\hline
\textsf{P} & $A\rhd B\to\square(A\rhd B)$ & + & + & + & + \\
\hline
\textsf{P\textsubscript{0}} & $A\rhd\Diamond B\to\square(A\rhd B)$ & - & + & / & + \\
\hline
\textsf{R} & $A\rhd B\to \neg(A\rhd\neg C)\rhd  B\wedge\square C$ & ? & + & ? & + \\
\hline
\textsf{W} & $A\rhd B\to A\rhd B\wedge\square\neg A$ & + & + & + & + \\
\hline
\textsf{W}$^*$ & $A\rhd B\to B\wedge\square C\rhd B\wedge\square C\wedge\square\neg A$ & + & + & ? & +\\
\hline
\end{tabular}
\end{center}

\vskip 2ex
De~Jongh and Veltman proved the completeness of the logics \il{}, \il{M} and \il{P}  w.r.t.\ their characteristic classes of ordinary (and finite) Veltman frames in \cite{deJongh-Veltman-90}.
As is usual for extension of the provability logic \textbf{GL}, all completeness proofs suffer from compactness-related issues. 
One way to go about this is to define a (large enough) adequate set of formulas and let worlds be maximal consistent subsets of such sets (used e.g. in \cite{deJongh-Veltman-90}). With interpretability logics and ordinary semantics, worlds have not been identified (just) with sets of formulas, because of issues described in e.g. \cite{deJongh-Veltman-90}. In  \cite{deJongh-Veltman-99}, de Jongh and Veltman proved completeness of the logic \il{W}  w.r.t.\ its characteristic class of ordinary (and finite) Veltman frames.

\vskip 2ex
Goris and Joosten introduced a more robust approach to proving completeness of interpretability logics, the \textit{construction method} \cite{Goris-Joosten-08}, \cite{Goris-Joosten-11}. 
In this type of proofs, one builds models step by step, and the final model is retrieved as a union. 
While closer to intuitions and more informative than the standard proofs, these proofs are hard to produce and verify due to their size. (They might have been shorter if tools from \cite{Bilkova-Goris-Joosten} have been used from the start.) 
For the purpose for which they were invented (completeness of \ilmn{} and \il{W*} w.r.t.\ ordinary semantics) they are still the only known tools. 

\vskip 2ex
The completeness of interpretability logics w.r.t.\ generalized semantics is usually an easy consequence of the 
completeness w.r.t.\ ordinary semantics.
In \cite{Perkov-Vukovic-16} and \cite{Mikec-Perkov-Vukovic-17}, filtration technique was used to prove the 
finite model property of \il{} and its extensions \il{M}, \ilmn \ and \il{W}$^*$ w.r.t.\ generalized 
Veltman semantics. 
Generalized semantics was used because issues occurred when merging worlds of ordinary Veltman models. 
Those explorations yielded some decidability results.

\vskip 2ex
The aim of this paper is to show completeness (w.r.t.\ generalized semantics) and decidability of some interpretability logics. We introduce a very direct type of proof for proving completeness; similar to \cite{deJongh-Veltman-90} in their general approach. We use \textit{smart labels} from \cite{Bilkova-Goris-Joosten} for this purpose. 
An example that illustrates benefits of using generalized semantics will be given in the section dedicated to \ilmn.

\vskip 2ex
The main new results of this paper are completeness and finite model property (and thus decidability) of \il{R} and \il{P\textsubscript{0}}. The principle \textsf{R} is important because it forms the basis of the, at the moment, strongest candidate for \il{(All)}. Results concerning the principle \il{P\textsubscript{0}} are interesting in a different way; they answer an old question: is there an unravelling technique that transforms generalized \il{X}-models to ordinary \il{X}-models, that preserves satisfaction of relevant characteristic properties? The answer is \textit{no}: we find \il{P\textsubscript{0}} to be complete w.r.t.\ generalized semantics, but it is known to be incomplete w.r.t.\ ordinary semantics.

Other results include reproving some known facts with (much) shorter proofs. 
Of particular interest is the logic \il{W}, which was known to be complete and decidable, 
    but for which we nevertheless reprove completeness w.r.t.\ gerenalized semantics using our approach.
We will explain our motivation for doing so in the section dedicated to \il{W}.

\section{Completeness w.r.t.\ generalized semantics}

In what follows, ``formula'' will always mean ``modal formula''. If the ambient logic in some context is \il{X}, a maximal consistent set w.r.t.\ \il{X} will be called an \il{X}-MCS. Let us now introduce \textit{smart labels} from \cite{Bilkova-Goris-Joosten}.

\begin{definition}[\cite{Bilkova-Goris-Joosten}, slightly modified Definition 3.1]
    \label{def-smart-label}
    Let $w$, $u$ and $x$ be some \il{X}-MCS's, and let $S$ and $T$ be arbitrary sets of 
    formulas.
    We define $w \prec_S u$ if for any finite $S' \subseteq S$ and any formula $A$ 
    we have that $A \rhd \bigvee_{G\in S'}\neg G \in w$ implies 
    $\neg A, \square \neg A \in u.$
\end{definition}
Note that the small differences between our Definition \ref{def-smart-label} and Definition 3.1 \cite{Bilkova-Goris-Joosten} do not affect results of \cite{Bilkova-Goris-Joosten} that we use.

\vskip 2ex
\noindent
\begin{definition}[\cite{Bilkova-Goris-Joosten}, page 4]
    \label{def-smart-label-sets}
    Let $u$ be an \il{X}-MCS, and $S$ an arbitrary set of formulas. Put:
    \begin{align*}
    	u_S^\square &= \{ \square \neg A : S' \subseteq S, \ S' \text{ finite}, \ 
    	A \rhd\bigvee_{G\in S'}  \neg G\in u \}; \\
       	u_S^\boxdot &= \{ \neg A, \square \neg A : S' \subseteq S, \ S' \text{ finite}, \ 
       	A \rhd\bigvee_{G\in S'}  \neg G\in u \}.
   \end{align*}
\end{definition}
Thus, $w \prec_S u$ if and only if $w_S^\boxdot \subseteq u$.
If $S = \emptyset$ then $u_\emptyset^\square = \{ \square \neg A : A \rhd \bot \in u \}$. 
Since $u$ is maximal consistent, usages of this set usually amount to the same as of 
$\{ \square A : \square A \in u \}.$ 

We will usually write $w \prec u$ instead of $w \prec_\emptyset u$.
\begin{lemma}[\cite{Bilkova-Goris-Joosten}, Lemma 3.2]
\label{lema3.2}
Let $w$, $u$ and $v$ be some \il{X}-MCS's, and let $S$ and $T$ be some sets of formulas. Then we have:
\begin{itemize}
    \item[a)] if $S\subseteq T$ and $w\prec_T u$, then $w\prec_S u$;
    \item[b)] if $w\prec_S u\prec v$, then $w\prec_S v$;
    \item[c)] if $w\prec_S u$, then $ S\subseteq u$.
\end{itemize}
\end{lemma}
We will tacitly use properties of the preceding lemma in most of our proofs.

The following two lemmas can be used to construct (or in our case, find) a MCS with the required properties.
\begin{lemma}[\cite{Bilkova-Goris-Joosten}, Lemma 3.4]
\label{problemi}
Let $w$ be an \il{X}-MCS, and let $\neg (B\rhd C)\in w.$ Then there is an \il{X}-MCS $u$ 
such that $w \prec_{\{\neg C\}} u$ and $B, \square \neg B\in u.$
\end{lemma}
\begin{lemma}[\cite{Bilkova-Goris-Joosten}, Lemma 3.5]
\label{nedostaci}
Let $w$ and $u$ be some \il{X}-MCS's such that $B\rhd C\in w,$ $w\prec_S u$ and $B\in u.$
Then there is an \il{X}-MCS $v$ such that $w\prec_S v$ and $C,\square\neg C\in v.$
\end{lemma}

Let $B$ be a formula, and $w$ a world in a generalized Veltman model.
Put $[B]_w = \{ u : wRu$ and $u\Vdash B\}$.

In the remainder of the current paper, we will assume that $\mathcal{D}$ is always a finite set of formulas, closed under taking subformulas and single negations, and $\top\in \mathcal{D}$.
The following definition is central to most of the results of this paper.

\begin{definition}
Let \textsf{X} be a  subset of $\{$\textsf{M}, \textsf{M\textsubscript{0}}, \textsf{P}, \textsf{P\textsubscript{0}}, \textsf{R}$\}$.
We say that $\mathfrak{M} = (W, R, \{S_w : w \in W\}, \Vdash)$ is the \il{X}-structure for 
a set of formulas $\mathcal{D}$ if:
\begin{align*}
           W &= \{ w : w \text{ is an \il{X}-MCS and for some } G \in \mathcal{D}, \ G 
           \wedge \square \neg G \in w \};\\
            wRu &\Leftrightarrow w \prec u;\\  
            uS_w V &\Leftrightarrow wRu, \ V \subseteq R[w], (\forall S)(w \prec_S u \Rightarrow 
            (\exists v \in V) w \prec_S v );\\
            w\Vdash p&\Leftrightarrow p\in w.
    \end{align*}
\label{ilx-struktura}    
\end{definition}

\begin{lemma}
If $\il{X} \nvdash \neg A$ then there is an \il{X}-MCS $w$ such that $A\wedge \square\neg A\in w.$
\label{lema-nepraznost}
\end{lemma}
\begin{proof}
 We are to show that $\{ A\wedge \square\neg A\}$ is an \il{X}-consistent set.   
 Suppose $A, \square\neg A\vdash \bot.$ 
 It follows that $\vdash \square\neg A\rightarrow \neg A.$ 
 Applying generalization (necessitation) gives $\vdash \square(\square\neg A\rightarrow \neg A ).$
 The L\"{o}b axiom implies $\vdash\square\neg A.$
 Now,  $\vdash\square\neg A$ and $A,\square\neg A\vdash\bot$ imply $A\vdash\bot,$
 i.e.\ $\vdash \neg A,$ a contradiction.
\end{proof}

\begin{lemma}
Let \textsf{X} be a subset of $\{$\textsf{M}, \textsf{M\textsubscript{0}}, \textsf{P}, 
\textsf{P\textsubscript{0}}, \textsf{R}$\}$.
The \il{X}-structure $\mathfrak{M}$ for a set formula $\mathcal{D}$ is a generalized Veltman model.
Furthermore, the following holds:
$$ \mathfrak{M},w\Vdash G \ \mbox{ if and only if }\ G\in w,$$
for all $G\in\mathcal{D}$ and $w\in W.$ 
\label{lemma-main}
\end{lemma}
\begin{proof}
Let us verify that the \il{X}-structure 
$\mathfrak{M} = (W, R, \{ S_w : w \in W \}, \Vdash )$ 
for $\mathcal{D}$ is a generalized Veltman model. 
Since $\il{X} \nvdash \bot$ and $\top\in\mathcal{D}$, Lemma \ref{lema-nepraznost} implies 
$W\neq \emptyset$ .

Transitivity of $R$ is immediate.
To see converse well-foundedness, assume there are more than $|\mathcal{D}|$ worlds in an 
$R$-chain. 
Then there are $x$ and $y$ with $xRy$ and for some $G \in \mathcal{D}$, 
$G, \square \neg G \in x, y$. 
However, $\square \neg G \in x$ and $G \in y$ obviously contradict the assumption that $xRy$ 
($x \prec y$).

Next, let us prove properties of $S_w$ for $w \in W$. Clearly $S_w \subseteq R[w] \times 
\mathcal{P}(R[w])$. If $xS_w V$, then $w \prec_\emptyset x$ implies there is at least one element $v$ in 
$V$ (with $w \prec_\emptyset v$). 
Quasi-reflexivity and monotonicity are obvious. Next, assume $wRxRu$ and $w \prec_S x$. 
Lemma \ref{lema3.2} and $w \prec_S x \prec u$ imply $w \prec_S u.$ Thus $x S_w \{u\}$. 
It remains to prove quasi-transitivity. Assume $x S_w V$ and $v S_w U_v$ for all $v \in V$. 
Put $U = \bigcup_v U_v$. We claim that $x S_w U$.  We have $U \subseteq R[w]$. 
Assume $w \prec_S x$. This and $xS_w V$ imply there is $v \in V$ such that $w \prec_S v$. 
This and $v S_w U_v$ imply there is $u \in U_v$ (thus also $u \in U$) such that $w \prec_S u$.

\vskip 2ex
Let us prove the truth lemma with respect to formulas contained in $\mathcal{D}$.
The claim is proved by induction on the complexity of $G\in\mathcal{D}.$
We will only consider the case $G=B\rhd C.$

Assume $B \rhd C \in w,$ \ $wR u$ and $u\Vdash B.$ 
Induction hypothesis implies $B\in u.$
We claim that $u S_w [C]_w$. Clearly $[C]_w \subseteq R[w]$. 
Assume $w \prec_S u$. Lemma \ref{nedostaci} implies there is an \il{X}-MCS $v$ with $w \prec_S v$
and $C, \square \neg C\in v$ (thus also $wRv$ and $v \in W$).
Induction hypothesis implies ${\mathfrak M},v\Vdash C$.

\vskip 2ex
To prove the converse, assume $B \rhd C \notin w$. Lemma \ref{problemi} implies there is $u$ with 
$w \prec_{\{\neg C\}} u$ and  $B, \square \neg B\in u$ (thus $u \in W$). 
It is immediate that $wRu$ and the induction hypothesis implies that $u \Vdash B$.
Assume $u S_w V.$ We are to show that $V \nVdash C$.
Since $w\prec_{\{ \neg C\}} u$ and $uS_w V$,  there is $v \in V$ 
such that $w \prec_{\{\neg C\}} v$. Lemma \ref{lema3.2} implies 
$\neg C\in v.$ 
The induction hypothesis implies $v \nVdash C$; thus $V \nVdash C$.
\end{proof}

\begin{theorem}
\label{glavni} 
Let $X \subseteq \{$\textsf{M}, \textsf{M\textsubscript{0}}, \textsf{P}, 
\textsf{P\textsubscript{0}}, \textsf{R}$\}$.
Assume that for every set $\mathcal{D}$ the \il{X}-structure for $\mathcal{D}$ possesses 
the property \kgen{X}.
Then \il{X} is complete w.r.t.\ \ilgen{X}-models.
\end{theorem}

\begin{proof}
Let $A$ be a formula such that $\nvdash \neg A$. Lemma \ref{lema-nepraznost} implies there is 
an \il{X}-MCS $w$ such that $A \wedge \square \neg A \in w.$ 
Let $\mathcal D$ have the usual properties, and contain $A$.
Let ${\mathfrak M}=(W, R, \{S_w : w \in W\}, \Vdash)$ be the \il{X}-structure for 
$\mathcal{D}$. 
Since $A \wedge \square \neg A \in w$  and $A \in \mathcal{D}$, we have $w \in W$.
Lemma \ref{lemma-main} implies $\mathfrak{M},w\nVdash \neg A.$ 
\end{proof}

\begin{corollary}
The logic \il{} is complete w.r.t.\ generalized Veltman models.
\end{corollary}

Note that the method presented in \cite{Vukovic08} now implies 
completeness of \il{} w.r.t.\ ordinary Veltman models.
Unfortunately, this method does not preserve characteristic properties in general.

In the following sections we prove (or reprove) the completeness of the following logics w.r.t.\ generalized semantics: \il{M}, \il{M\textsubscript{0}}, \il{P}, \il{P\textsubscript{0}},  \il{R}, \il{W} and \il{W*}.

\subsection{The logic \il{M}}

Completeness of the logic \il{M} w.r.t.\ generalized semantics is an easy consqeuence of the completeness 
of \il{M} w.r.t.\ ordinary semantics, first proved by de Jongh and Veltman \cite{deJongh-Veltman-90}. 
Another proof of the same result was given by Goris and Joosten, using the construction method in \cite{Goris-Joosten-11}.

Verbrugge determined the characteristic property \kgen{M} in 1992.\ in an unpublished paper:
$$ uS_w V \Rightarrow (\exists V' \subseteq V)( uS_w V' \ \& \ R[V'] 
\subseteq R[u]).$$

\begin{lemma}[\cite{Bilkova-Goris-Joosten}, Lemma 3.7]
\label{lema-ILM}
Let $w$ and $u$ be some \il{M}-MCS's, and let $S$ be a set of formulas.  
If $w \prec_S u$ then $w \prec_{S \cup u_\emptyset^\square} u$.
\end{lemma}

\begin{theorem}
The logic \il{M} is complete w.r.t.\ \ilgen{M}-models.
\end{theorem}
\begin{proof}
Given Theorem \ref{glavni}, it suffices to show that for any set $\mathcal{D}$, the \il{M}-structure for $\mathcal D$ possesses the property \kgen{M}. 
Let $(W,R,\{ S_w:w\in W\},\Vdash)$ be the \il{M}-structure for $\mathcal{D}.$

Let $uS_w V$ and take $V' = \{ v\in V : w \prec_{u_\emptyset^\square} v \}$. 
We claim $uS_w V'$ and $R[V']\subseteq R[u].$
Suppose $w \prec_S u$. Lemma \ref{lema-ILM} implies 
$w \prec_{S \cup u_\emptyset^\square} u.$
Since $uS_w V$, we have that there is $v \in V$ with $w \prec_{S \cup u_\emptyset^\square} v.$ 
So, $v \in V'$. Thus, $u S_w V'$. 

Now let $v \in V'$ and $z\in W$ be such that $vRz$. 
Since $v \in V'$, $w \prec_{u_\emptyset^\square} v$. 
Then for all $\square B \in u$ we have $\square B \in v.$ 
Since $vRz$, we have $B, \square B \in z$. 
So, $u \prec z$ i.e.\  $uRz$.
\end{proof}

\subsection{The logic \ilmn{}}

Modal completeness of \ilmn{} w.r.t.\ Veltman semantics was proved in \cite{Goris-Joosten-08} by E.\ Goris and 
J.\ J.\ Joosten. Certain difficulties encountered in this proof were our main motivation for using generalized Veltman semantics. 
We will sketch these difficulties and show in what way does generalized semantics overcome them. 

\vskip 2ex
Characteristic property \kgen{M\textsubscript{0}}  (see \cite{Mikec-Perkov-Vukovic-17}):
$$ wRuRxS_wV\ \Rightarrow \ (\exists V'\subseteq V) (uS_wV' \ \& \ R[V']\subseteq R[u])).$$

\begin{lemma}[\cite{Bilkova-Goris-Joosten}, Lemma 3.9]
\label{label-mn}
Let $w$, $u$ and $x$ be \ilmn{}-MCS's, and $S$ an arbitrary set of formulas. 
If $w \prec_S u \prec x$ then $w \prec_{S \cup u_\emptyset^\square} x$.
\end{lemma}

To motivate our proving of completeness (of \ilmn, but also in general) w.r.t.\ generalized semantics, let us sketch a situation for which there are clear benefits in working with generalized semantics. We do this only now because \il{M\textsubscript{0}} is sufficiently complex to display (some of) these benefits.
Suppose we are building models step by step (as in the \textit{construction method} \cite{Goris-Joosten-08}), and worlds $w$, $u_1$, 
$u_2$ and $x$ occur in the configuration displayed in Figure \ref{fig:slika}.
Furthermore, suppose we need to produce an $S_w$-successor $v$ of $x$. 

\begin{figure}[h!]
    \centering

\tikzset{every picture/.style={line width=0.75pt}} 

\begin{tikzpicture}[x=0.75pt,y=0.75pt,yscale=-1,xscale=1]

\draw    (76.5,147.33) -- (126.84,113.45) ;
\draw [shift={(128.5,112.33)}, rotate = 506.06] [color={rgb, 255:red, 0; green, 0; blue, 0 }  ][line width=0.75]    (10.93,-3.29) .. controls (6.95,-1.4) and (3.31,-0.3) .. (0,0) .. controls (3.31,0.3) and (6.95,1.4) .. (10.93,3.29)   ;
\draw [shift={(76.5,147.33)}, rotate = 326.06] [color={rgb, 255:red, 0; green, 0; blue, 0 }  ][fill={rgb, 255:red, 0; green, 0; blue, 0 }  ][line width=0.75]      (0, 0) circle [x radius= 3.35, y radius= 3.35]   ;
\draw    (76.5,147.33) -- (126.81,179.26) ;
\draw [shift={(128.5,180.33)}, rotate = 212.4] [color={rgb, 255:red, 0; green, 0; blue, 0 }  ][line width=0.75]    (10.93,-3.29) .. controls (6.95,-1.4) and (3.31,-0.3) .. (0,0) .. controls (3.31,0.3) and (6.95,1.4) .. (10.93,3.29)   ;

\draw    (128.5,112.33) -- (177.14,147.98) ;
\draw [shift={(178.75,149.17)}, rotate = 216.24] [color={rgb, 255:red, 0; green, 0; blue, 0 }  ][line width=0.75]    (10.93,-3.29) .. controls (6.95,-1.4) and (3.31,-0.3) .. (0,0) .. controls (3.31,0.3) and (6.95,1.4) .. (10.93,3.29)   ;
\draw [shift={(128.5,112.33)}, rotate = 36.24] [color={rgb, 255:red, 0; green, 0; blue, 0 }  ][fill={rgb, 255:red, 0; green, 0; blue, 0 }  ][line width=0.75]      (0, 0) circle [x radius= 3.35, y radius= 3.35]   ;
\draw    (128.5,180.33) -- (177.05,150.22) ;
\draw [shift={(178.75,149.17)}, rotate = 508.19] [color={rgb, 255:red, 0; green, 0; blue, 0 }  ][line width=0.75]    (10.93,-3.29) .. controls (6.95,-1.4) and (3.31,-0.3) .. (0,0) .. controls (3.31,0.3) and (6.95,1.4) .. (10.93,3.29)   ;
\draw [shift={(128.5,180.33)}, rotate = 328.19] [color={rgb, 255:red, 0; green, 0; blue, 0 }  ][fill={rgb, 255:red, 0; green, 0; blue, 0 }  ][line width=0.75]      (0, 0) circle [x radius= 3.35, y radius= 3.35]   ;
\draw  [dash pattern={on 4.5pt off 4.5pt}]  (178.75,149.17) .. controls (218.35,119.47) and (217.52,176.18) .. (256.31,148.21) ;
\draw [shift={(257.5,147.33)}, rotate = 503.13] [color={rgb, 255:red, 0; green, 0; blue, 0 }  ][line width=0.75]    (10.93,-3.29) .. controls (6.95,-1.4) and (3.31,-0.3) .. (0,0) .. controls (3.31,0.3) and (6.95,1.4) .. (10.93,3.29)   ;
\draw [shift={(178.75,149.17)}, rotate = 323.13] [color={rgb, 255:red, 0; green, 0; blue, 0 }  ][fill={rgb, 255:red, 0; green, 0; blue, 0 }  ][line width=0.75]      (0, 0) circle [x radius= 3.35, y radius= 3.35]   ;
\draw    (319.5,147.33) -- (369.84,113.45) ;
\draw [shift={(371.5,112.33)}, rotate = 506.06] [color={rgb, 255:red, 0; green, 0; blue, 0 }  ][line width=0.75]    (10.93,-3.29) .. controls (6.95,-1.4) and (3.31,-0.3) .. (0,0) .. controls (3.31,0.3) and (6.95,1.4) .. (10.93,3.29)   ;
\draw [shift={(319.5,147.33)}, rotate = 326.06] [color={rgb, 255:red, 0; green, 0; blue, 0 }  ][fill={rgb, 255:red, 0; green, 0; blue, 0 }  ][line width=0.75]      (0, 0) circle [x radius= 3.35, y radius= 3.35]   ;
\draw    (319.5,147.33) -- (369.81,179.26) ;
\draw [shift={(371.5,180.33)}, rotate = 212.4] [color={rgb, 255:red, 0; green, 0; blue, 0 }  ][line width=0.75]    (10.93,-3.29) .. controls (6.95,-1.4) and (3.31,-0.3) .. (0,0) .. controls (3.31,0.3) and (6.95,1.4) .. (10.93,3.29)   ;

\draw    (371.5,112.33) -- (420.14,147.98) ;
\draw [shift={(421.75,149.17)}, rotate = 216.24] [color={rgb, 255:red, 0; green, 0; blue, 0 }  ][line width=0.75]    (10.93,-3.29) .. controls (6.95,-1.4) and (3.31,-0.3) .. (0,0) .. controls (3.31,0.3) and (6.95,1.4) .. (10.93,3.29)   ;
\draw [shift={(371.5,112.33)}, rotate = 36.24] [color={rgb, 255:red, 0; green, 0; blue, 0 }  ][fill={rgb, 255:red, 0; green, 0; blue, 0 }  ][line width=0.75]      (0, 0) circle [x radius= 3.35, y radius= 3.35]   ;
\draw    (371.5,180.33) -- (420.05,150.22) ;
\draw [shift={(421.75,149.17)}, rotate = 508.19] [color={rgb, 255:red, 0; green, 0; blue, 0 }  ][line width=0.75]    (10.93,-3.29) .. controls (6.95,-1.4) and (3.31,-0.3) .. (0,0) .. controls (3.31,0.3) and (6.95,1.4) .. (10.93,3.29)   ;
\draw [shift={(371.5,180.33)}, rotate = 328.19] [color={rgb, 255:red, 0; green, 0; blue, 0 }  ][fill={rgb, 255:red, 0; green, 0; blue, 0 }  ][line width=0.75]      (0, 0) circle [x radius= 3.35, y radius= 3.35]   ;
\draw  [dash pattern={on 4.5pt off 4.5pt}]  (421.75,149.17) .. controls (461.35,119.47) and (460.52,176.18) .. (499.31,148.21) ;
\draw [shift={(500.5,147.33)}, rotate = 503.13] [color={rgb, 255:red, 0; green, 0; blue, 0 }  ][line width=0.75]    (10.93,-3.29) .. controls (6.95,-1.4) and (3.31,-0.3) .. (0,0) .. controls (3.31,0.3) and (6.95,1.4) .. (10.93,3.29)   ;
\draw [shift={(421.75,149.17)}, rotate = 323.13] [color={rgb, 255:red, 0; green, 0; blue, 0 }  ][fill={rgb, 255:red, 0; green, 0; blue, 0 }  ][line width=0.75]      (0, 0) circle [x radius= 3.35, y radius= 3.35]   ;
\draw   (500.5,147.33) .. controls (500.51,108.56) and (531.95,77.13) .. (570.73,77.13) .. controls (609.5,77.14) and (640.93,108.58) .. (640.93,147.36) .. controls (640.92,186.14) and (609.48,217.57) .. (570.7,217.56) .. controls (531.92,217.55) and (500.49,186.11) .. (500.5,147.33) -- cycle ;
\draw   (543.92,114.34) .. controls (543.92,100.54) and (555.11,89.34) .. (568.92,89.34) .. controls (582.72,89.34) and (593.92,100.54) .. (593.92,114.34) .. controls (593.92,128.15) and (582.72,139.34) .. (568.92,139.34) .. controls (555.11,139.34) and (543.92,128.15) .. (543.92,114.34) -- cycle ;
\draw   (545.71,180.35) .. controls (545.71,166.54) and (556.91,155.35) .. (570.71,155.35) .. controls (584.52,155.35) and (595.71,166.54) .. (595.71,180.35) .. controls (595.71,194.15) and (584.52,205.35) .. (570.71,205.35) .. controls (556.91,205.35) and (545.71,194.15) .. (545.71,180.35) -- cycle ;
\draw  [dash pattern={on 4.5pt off 4.5pt}]  (371.5,112.33) .. controls (417.27,29.75) and (481.85,167.94) .. (542.99,115.15) ;
\draw [shift={(543.92,114.34)}, rotate = 498.16] [color={rgb, 255:red, 0; green, 0; blue, 0 }  ][line width=0.75]    (10.93,-3.29) .. controls (6.95,-1.4) and (3.31,-0.3) .. (0,0) .. controls (3.31,0.3) and (6.95,1.4) .. (10.93,3.29)   ;
\draw [shift={(371.5,112.33)}, rotate = 299] [color={rgb, 255:red, 0; green, 0; blue, 0 }  ][fill={rgb, 255:red, 0; green, 0; blue, 0 }  ][line width=0.75]      (0, 0) circle [x radius= 3.35, y radius= 3.35]   ;
\draw  [dash pattern={on 4.5pt off 4.5pt}]  (371.5,180.33) .. controls (437.17,253.96) and (463.24,139.49) .. (544.48,179.73) ;
\draw [shift={(545.71,180.35)}, rotate = 207.07] [color={rgb, 255:red, 0; green, 0; blue, 0 }  ][line width=0.75]    (10.93,-3.29) .. controls (6.95,-1.4) and (3.31,-0.3) .. (0,0) .. controls (3.31,0.3) and (6.95,1.4) .. (10.93,3.29)   ;
\draw [shift={(371.5,180.33)}, rotate = 48.27] [color={rgb, 255:red, 0; green, 0; blue, 0 }  ][fill={rgb, 255:red, 0; green, 0; blue, 0 }  ][line width=0.75]      (0, 0) circle [x radius= 3.35, y radius= 3.35]   ;

\draw (257.5,147.33) node   {$\CIRCLE $};
\draw (76,163) node   {$w$};
\draw (113,102) node   {$u_{1}$};
\draw (113.5,187) node   {$u_{2}$};
\draw (125,132.5) node   {$\square B_{1}$};
\draw (125,161.5) node   {$\square B_{2}$};
\draw (368,132.5) node   {$\square B_{1}$};
\draw (368,161.5) node   {$\square B_{2}$};
\draw (261,131.5) node   {$\square B_{1} ,\ \square B_{2}$};
\draw (610,126.5) node   {$\square B_{1}$};
\draw (610,167.5) node   {$\square B_{2}$};
\draw (179,164) node   {$x$};
\draw (319,164) node   {$w$};
\draw (356,103) node   {$u_{1}$};
\draw (356.5,186) node   {$u_{2}$};
\draw (422,165) node   {$x$};
\draw (257,165) node   {$v$};
\draw (510,146) node   {$V$};
\draw (557,180) node   {$V_{2}$};
\draw (556,114) node   {$V_{1}$};

\end{tikzpicture}

\caption{Left: extending an ordinary Veltman model. Right: extending a generalized Veltman model. Straight lines represent $R$-transitions, while curved lines represent $S_w$-transitions. Full lines represent the starting configuration, and dashed lines represent transitions that are to be added. }
\label{fig:slika}
\end{figure}
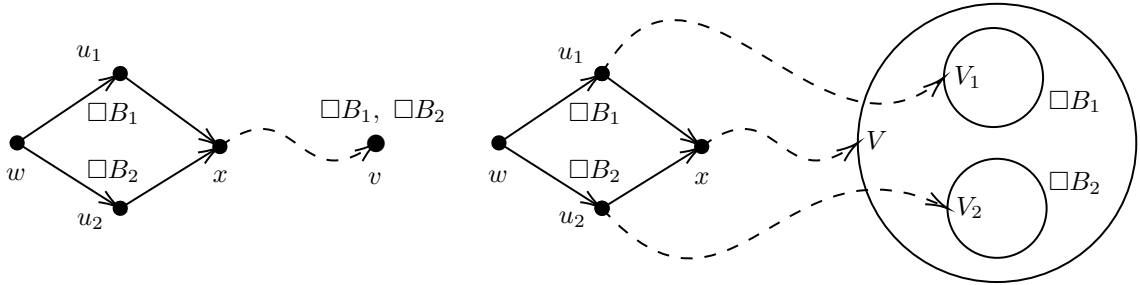

With ordinary semantics, we need to ensure that for our $S_w$-successor $v$, for each 
$\square B_1 \in u_1$ and $\square B_2 \in u_2$, we have $\square B_1, \square B_2 \in v$. 
It is not obvious that such construction is possible. 
In case of \ilmn, it was successfully solved in \cite{Goris-Joosten-08} by preserving the invariant 
that sets of boxed formulas in $u_i$ are linearly ordered. 
This way, finite (quasi-)models can always be extended by only looking at the last $u_i$. 
With generalized semantics, we need to produce a whole set of worlds $V$, but the requirements on 
each particular world are less demanding. 
For each $u_i$, there has to be a corresponding $V_i \subseteq V$ with $\square B_i$ contained 
(true) in every world of $V_i$. 
Lemma \ref{label-mn} gives a recipe for producing such worlds.

\begin{theorem}
The logic \ilmn{} is complete w.r.t.\ \ilgen{M\textsubscript{0}}-models.
\end{theorem}
\begin{proof}
Given Theorem \ref{glavni}, it suffices to show that for any set $\mathcal{D}$, the \ilmn-structure for $\mathcal{D}$ possesses the property \kgen{M\textsubscript{0}}.
Let $(W,R,\{ S_w:w\in W\},\Vdash)$ be the \ilmn{}-structure for $\mathcal{D}.$

Assume $wRuRx S_w V$ and take $V' = \{ v \in V : w \prec_{u_\emptyset^\square} v \}$. 
We claim that $u S_w V'$ and $R[V'] \subseteq R[u]$.
Obviously $V'\subseteq V\subseteq R[w].$ 
Assume $w \prec_S u$. Lemma \ref{label-mn} and $w \prec_{S} u \prec x$ imply 
$w \prec_{S \cup u_\emptyset^\square} x$. 
Now $xS_wV$ and the definition of $S_w$ imply there is $v \in V$ such that
$w \prec_{S \cup u_\emptyset^\square} v$.
Lemma \ref{lema3.2} implies $w\prec_{u^\square_\emptyset} v.$ So, $v\in V'.$

It remains to verify that $R[V'] \subseteq R[u]$. Let $v \in V'$ and $z\in W$ be worlds such 
that $vRz$. 
Since $w \prec_{u_\emptyset^\square} v$, for all $\square B \in u$ we have $\square B \in v$, 
and since $vRz$, it follows that $\square B, B \in z$. Thus, $u \prec z$ i.e.\  $uRz.$
\end{proof}

\subsection{The logic \il{P}}

As in the case of the logic \il{M}, the completeness of \il{P} w.r.t.\ generalized semantics is an easy consqeuence of the completeness 
of \il{P} w.r.t.\ ordinary semantics, first proved by de Jongh and Veltman \cite{deJongh-Veltman-90}.

Verbrugge determined the characteristic property \kgen{P} in 1992.\ in an unpublished paper:
\[ wRw'RuS_w V \Rightarrow (\exists V' \subseteq V)\ uS_{w'} V'  \]

\begin{lemma}[\cite{Bilkova-Goris-Joosten} Lemma 3.8]
\label{lema-ILP}
Let $w$, $x$ and $u$ be some \il{P}-MCS's, and let $S$ and $T$ be arbitrary sets of formulas.  
If  $w \prec_S x \prec_T u$ then $w \prec_{S \cup x_T^\boxdot} u$.
\end{lemma}
\begin{theorem}
The logic \il{P} is complete w.r.t.\ \ilgen{P}-models.
\end{theorem}
\begin{proof}
Given Theorem \ref{glavni}, it suffices to show that for any set $\mathcal{D}$, the \il{P}-structure for $\mathcal D$ possesses the property \kgen{P}. 
Let $(W,R,\{ S_w:w\in W\},\Vdash)$ be the \il{P}-structure for $\mathcal{D}.$

Let $wRw'RuS_w V$ and take $V' = V \cap R[w']$. We claim $uS_{w'} V'.$ 
Let $T$ be arbitrary such that $w'\prec_T u$. 
Lemma \ref{lema-ILP} and $w\prec_\emptyset w'\prec_T u$ imply $w \prec_{{w'}_T^\boxdot} u.$ 
Now, $uS_w V$ implies that there is a $v \in V$ with $w \prec_{{w'}_T^\boxdot} v$. 
Let $A \rhd \neg\bigwedge T' \in w'$ for some finite $T' \subseteq T$. 
Then $\neg A, \square \neg A \in {w'}_T^\boxdot$. 
Lemma \ref{lema3.2}  and $w \prec_{{w'}_T^\boxdot} v$ imply $\neg A, \square \neg A \in v$.
Thus $w' \prec_T v$. 
Finally, $V' \subseteq R[w']$ holds by assumption, thus $uS_{w'} V'$.
\end{proof}

\subsection{The logic \il{P\textsubscript{0}}}

The interpretability principle $\textsf{P\textsubscript{0}} = A \rhd \Diamond B \to \square(A \rhd B)$ is introduced 
in J.\ J.\ Joosten's master thesis in 1998.

In \cite{Goris-Joosten-11} it is shown 
that the interpretability logic \il{P\textsubscript{0}} is incomplete w.r.t.\ Veltman models. 
Since we will show that \il{P\textsubscript{0}} is 
complete w.r.t.\ generalized semantics, this is the first example of 
an interpretability logic complete w.r.t.\ the generalized semantics, but incomplete w.r.t.\ ordinary semantics.

Characteristic property \kgen{P\textsubscript{0}} was determined in \cite{Goris-Joosten-11} . A slightly reformulated version:
\[
     wRxRuS_w V \ \& \ (\forall v \in V)  R[v] \cap Z \neq \emptyset \ \Rightarrow \ (\exists Z' \subseteq Z) uS_x Z'
\]


The following technical lemma is almost obvious.
\begin{lemma}
\label{psintaksa}
Let $x$ be an \il{X}-MCS, $A$ a formula, and $T$ a finite set of formulas. Let $B_G$ be an arbitrary formula, and $T_G$ an arbitrary finite set of formulas, for every $G\in T$. Furthermore, assume:
\begin{itemize}
    \item[a)]  $A\rhd \bigvee_{G\in T} B_G \in x$;
    \item[b)] $(\forall G\in T) \ B_G\rhd \bigvee_{H\in T_G} \neg H\in x$. 
\end{itemize}
Then we have $A\rhd \bigvee_{H \in S'} \neg H\in x$, where $S' = \bigcup_{G\in T} T_G$.
\end{lemma}
\begin{proof}
Let $G \in T$. Since $T_G\subseteq S'$, clearly
$\vdash \bigvee_{H\in T_G} \neg H\rhd \bigvee_{H\in S'}\neg H$.
The requirement b) and the axiom (J2) imply 
$B_G \rhd \bigvee_{H\in S'} \neg H\in x$.
Now $|T|-1$ applications of the axiom (J3) give 
$\bigvee_{G\in T} B_G \rhd \bigvee_{H\in S'}\neg  H\in x$.
Finally, apply the requirement a) and the axiom (J2).
\end{proof}

Next we need a labelling lemma for \il{P\textsubscript{0}}. This is where we use the technical lemma above.
\begin{lemma} 
\label{labelling-pn}
Let $w$, $x$ and $u$ be some \il{P\textsubscript{0}}-MCS's, and let $S$ be a set of formulas.
If $w  \prec x \prec_S u$ then $w \prec_{ x_S^\square } u$.
\end{lemma}
\begin{proof}
Let $A$ be an arbitrary formula. Let $T\subseteq x_S^\square$ be a finite set such that 
$A\rhd \bigvee_{G\in T} \neg G\in w.$ 
We will prove that $\neg A,\square \neg A\in u.$
If $G\in T \ (\subseteq x_S^\square)$, then $G=\square\neg B_G$, for some formula $B_G$.
Thus $A\rhd \bigvee_{G\in T} \neg \square\neg B_G\in w$, and by easy inferences and maximal consistency:
$A\rhd \bigvee_{G\in T} \Diamond B_G\in w$, and $A\rhd \Diamond \bigvee_{G\in T} B_G\in w$.
Applying \textsf{P\textsubscript{0}} gives
$\square (A\rhd\bigvee_{G\in T} B_G)\in w$.
The assumption $w\prec x$ implies $A\rhd \bigvee_{G\in T} B_G \in x.$
For each $G\in T \ (\subseteq x_S^\square)$ there is a finite subset $T_G$ of $S$ 
such that $B_G\rhd \bigvee_{H\in T_G} \neg H\in x.$
Let $S'=\bigcup_{G\in T} T_G.$ Clearly $S'$ is a finite subset of $S.$
Lemma \ref{psintaksa} implies $A\rhd \bigvee_{H\in S'} \neg H\in x.$ 
Finally, $S' \subseteq S$ and the assumption $x\prec_S u$ imply $\neg A,\square\neg A\in u.$
\end{proof}

The following simple observation is useful both for \il{P\textsubscript{0}}  and \il{R}. 

\begin{lemma}
\label{rsintaksa}
Let $w$, $x$, $v$ and $z$ be some \il{X}-MCS's, and let $S$ be a set of formulas.  
If $w \prec_{x_S^\square} v\prec z$ then $x\prec_S z$.
\end{lemma}
\begin{proof}
Let $S'$ be a finite subset of $S$ with $A \rhd \bigvee_{G\in S'} \neg G \in x$. 
Then $\square \neg A \in x_S^\square$. 
Now $w \prec_{x_S^\square} v$ and Lemma \ref{lema3.2} imply 
$\square \neg A \in v$. 
Since $v\prec z$, we have $\neg A, \square \neg A \in z$.
\end{proof}

\begin{theorem}\label{ILP0-complete}
The logic \il{P\textsubscript{0}} is complete w.r.t.\ \ilgen{P\textsubscript{0}}-frames.
\end{theorem}
\begin{proof}
Given Theorem \ref{glavni}, it suffices to show that for any set $\mathcal{D}$, the \il{P\textsubscript{0}}-structure for $\mathcal D$ possesses the property \kgen{P\textsubscript{0}}. 
Let $(W,R,\{ S_w:w\in W\},\Vdash)$ be the \il{P\textsubscript{0}}-structure for $\mathcal{D}.$

Assume $wRxRuS_w V$ and $R[v]\cap Z\neq \emptyset$ for each $v\in V$.
We will prove that there is $Z'\subseteq Z$ such that $uS_x Z'$.

Let $S$ be a set of formulas such that $w \prec x\prec_S u$.
Lemma \ref{labelling-pn} implies $w\prec_{x_S^\square} u.$
Since $uS_w V$, there is $v\in V$ such that $w\prec_{x_S^\square} v.$
Since $R[v]\cap Z\neq \emptyset$, choose a world $z_S\in R[v]\cap Z$.
Now $w\prec_{x_S^\square} v \prec z_S$ and Lemma \ref{rsintaksa} imply $x\prec_S z_S$. 
Put $Z'=\{ z_S : S$ is a set of formulas such that $x\prec_S u \}$.
Clearly $Z'\subseteq Z.$ 
So, $Z'\subseteq R[x]$, and since for each set $S$ such that $x\prec_S u$ we have $x\prec_S z_S $,  it follow that $uS_x Z'.$
\end{proof}

\vskip 2ex
In \cite{Vukovic08} a possibility was explored of transforming a generalized Veltman model to an ordinary Veltman model, such that these two models are bisimilar (in some aptly defined sense). A natural question is whether such transformation exists if we add the requirement that characteristic properties are preserved. The example of \il{P\textsubscript{0}} shows that there are \ilgen{P\textsubscript{0}}-models with no (bisimilar or otherwise) counterpart \il{P\textsubscript{0}}-models.

\subsection{The logic \il{R}}

Completeness of \il{R} w.r.t.\ ordinary Veltman semantics is an open problem (see \cite{Bilkova-Goris-Joosten}), but completeness w.r.t.\ the generalized semantics is not yet resolved either. 
In this section we will prove that \il{R} is complete w.r.t.\ the generalized semantics.

Characteristic property \kgen{R} was determined in \cite{Goris-Joosten-11}.  A slightly reformulated version:
\[
        wRxRuS_w V \Rightarrow (\forall C \in \mathcal{C}(x, u))(\exists U \subseteq V)(x S_w U \ \& \ 
        R[U] \subseteq C),
\]
where $\mathcal{C}(x, u) = \{ C \subseteq R[x] : (\forall Z) (uS_x Z\Rightarrow Z\cap C \neq 
\emptyset) \}$ is the family of ``choice sets''.

\vskip 2ex

\begin{lemma}[\cite{Bilkova-Goris-Joosten}, Lemma 3.10] 
\label{lema-ILR}
Let $w$, $x$ and $u$ be some \il{R}-MCS's, and let $S$ and $T$ be arbitrary sets of formulas. 
If $w \prec_S x \prec_T u$ then $w \prec_{S \cup x_T^\square} u.$
\end{lemma}

\begin{theorem}\label{ILR_complete}
The logic \il{R} is complete w.r.t.\ \ilgen{R}-models.
\end{theorem}
\begin{proof}
Given Theorem \ref{glavni}, it suffices to show that for any set $\mathcal{D}$, the \il{R}-structure for $\mathcal D$ possesses the property \kgen{R}. 
Let $(W,R,\{ S_w:w\in W\},\Vdash)$ be the \il{R}-structure for $\mathcal{D}.$

Assume $wRxRu S_w V$ and $C \in \mathcal{C}(x, u)$. 
We are to show that $(\exists U\subseteq V) (xS_w U  \ \& \ R[U]\subseteq C)$.
We will first prove an auxiliary claim: 
\[
    (\forall S) \big(w \prec_S x\ \Rightarrow \ (\exists v \in V) (w \prec_{S \cup 
    x_\emptyset^\square} v \ \& \ R[v] \subseteq C)\big) .
\]
So, let $S$ be arbitrary such that $w \prec_S x$, and suppose (for a contradiction) that for every 
$v \in V$ with $w \prec_{S \cup x_\emptyset^\square} v$, we 
have $R[v] \nsubseteq C$, that is, there is some $z_v \in R[v] \setminus C$. 
Let $Z = \{z_v : v \in V, w \prec_{S \cup x_\emptyset^\square} v \}.$ We claim that $u S_x Z$. 
Let $T$  be arbitrary such that $x \prec_T u$, and we should prove that there exists $z\in Z$ such 
that $x\prec_T z.$
From $w \prec_S x \prec_T u$ and Lemma \ref{lema-ILR} it 
follows that $w \prec_{S \cup x_T^\square} u$. Since $u S_w V$, there is $v \in V$ with 
$w \prec_{S \cup x_T^\square} v$. 
Now, $x_\emptyset^\square\subseteq x_T^\square$ and Lemma \ref{lema3.2} imply  
$w \prec_{S \cup x_\emptyset^\square} v$, so there is a world $z_v\in Z$ as defined earlier.
Furthermore, $w \prec_{x_T^\square} v\prec z_v$ and Lemma \ref{rsintaksa} imply $x \prec_T z_v$.
To prove $u S_x Z$ it remains to verify that $Z \subseteq R[x]$. Let $z_v \in Z$ be arbitrary and apply Lemma \ref{lema3.2} and Lemma \ref{rsintaksa} as before.
Now, $u S_x Z$ and $C \in \mathcal{C}(x, u)$ imply $C \cap Z \neq \emptyset$, contradicting the 
definition of $Z$. This concludes the proof of the auxiliary claim.

\vskip 2ex
Let $U=\{ v\in V: w\prec_{ x_\emptyset^\square} v \  
\mbox{ and } \ R[v]\subseteq C\}$.
Auxiliary claim implies $U \neq \emptyset$. If $w\prec_S x$, auxiliary claim implies there is 
 $v\in U$ such that $w\prec_{S\cup x_\emptyset^\square} v$ and $R[v] \subseteq C$, so $v \in U$. Thus $xS_w U$. It is clear that $R[U]\subseteq C$.
\end{proof}

\section{The logics \il{W} and \il{W}*}

To prove that \il{W} is complete, one could try to find a sufficiently strong ``labelling lemma'' and utilise Definition \ref{ilx-struktura}. One candidate might be the following condition:
\[
    w \prec_S u \ \Rightarrow \ (\exists G \in \mathcal{D}) w \prec_{S \cup \{ \square \neg G \} } u \text{ and } G \in u,
\]
where $\mathcal{D}$ is finite, closed under subformulas and such that each $w \in W$ contains $A_w$ and $\Box \neg A_w$ for some $A_w \in \mathcal D$. 

Since we weren't successful in finding a sufficiently strong labelling lemma for \il{W}, we will use a modified version of Definition \ref{ilx-struktura} to work with  \il{W} and its extensions. This way we won't require a labelling lemma, but we lose generality in the following sense. To prove the completeness of \il{XW}, for some $X$, it no longer suffices to simply show that the structure defined in Definition \ref{ilx-struktura} has the required characteristic property (when each world is an \il{X}-MCS). Instead, the characteristic property of \il{X} has to be shown to hold even on the modified structure. So, to improve compatibility with proofs based on Definition \ref{ilx-struktura}, we should prove the completeness of $\il{W}$ with as similar definition to Definition \ref{ilx-struktura} as possible. That is what we do in the remainder of this section. This approach turns out to be good enough for \il{W*} (\il{M\textsubscript{0}W}). We didn't succeed in using it to prove the completeness of \il{RW}. However, to the best of our knowledge, \il{RW} might not be complete at all. 

\vskip 2ex
In \cite{Goris-Joosten-11} the (complement of the) characteristic class for \il{W} is given by the condition Not-\textsf{W} such that for 
any generalized Veltman frame $\mathfrak{F}$ we have that
$$ \mathfrak{F}\models \text{Not-}\textsf{W} \ \text{ if and only if } \ 
\mathfrak{F}\nVdash \textsf{W}.$$
Another condition is \kgen{W} from \cite{Mikec-Perkov-Vukovic-17}:
\[
    uS_w V \Rightarrow (\exists V' \subseteq V)(u S_w V' \ \& \ R[V'] \cap 
    S_w^{-1}[V] = 
    \emptyset ). 
\]
We will use (this formulation of) \kgen W in what follows. 
In the proof of the completeness of logic \il{W} we will use the following two lemmas. In what follows, \il{WX} denotes an arbitrary extension of \il{W}.

\begin{lemma}[\cite{Bilkova-Goris-Joosten}, Lemma 3.12]
\label{problemi-W}
Let $w$ be an \il{WX}-MCS, and $B$ and $C$  formulas such that $\neg (B\rhd C)\in w.$ 
Then there is an \il{WX}-MCS $u$ such that $w\prec_{\{ \square\neg B,\neg C\}} u$ and  $B\in u.$
\end{lemma}

\begin{lemma}[\cite{Bilkova-Goris-Joosten}, Lemma 3.13]
\label{nedostaci-W}
Let $w$ and $u$ be some \il{WX}-MCS, $B$ and $C$ some formulas, and $S$ a set of formulas such that 
$B\rhd C\in w,$ $w\prec_S u$  and $B\in u.$
Then there is an \il{WX}-MCS $v$ such that $w\prec_{S\cup \{ \square\neg B\} } v$ and 
$C,\square\neg C\in v.$
\end{lemma}

Given a binary relation $R$, let $\dot{R}[x] =R[x]\cup \{ x\}.$ 


\begin{definition}
Let \textsf{X} be \textsf{W} or \textsf{W*}. We say that $\mathfrak{M} = (W, R, \{S_w : w \in W\}, \Vdash)$ is the \il{X}-structure for a set of formulas
$\mathcal{D}$ if:
        \begin{align*}
                W &= \{ w : w \text{ is an \il{X}-MCS and for some } G \in \mathcal{D}, 
                \ G \wedge \square \neg G \in w \};\\
                wRu &\Leftrightarrow w \prec u;\\  
                uS_w V &\Leftrightarrow wRu, V \subseteq R[w] \text{ and one of the following 
                holds:}\\ 
                & (a) \ \ \ V \cap \dot{R}[u] \neq \emptyset; \\
                & (b) \ \ \ (\forall S)(w \prec_S u \Rightarrow (\exists v \in V) (\exists G \in 
                \mathcal{D} \cap \bigcup\dot{R}[u]) \ w \prec_{S\cup \{\square\neg G\} } v);\\
                w\Vdash p&\Leftrightarrow p\in w.
        \end{align*}
\end{definition}

\begin{lemma}
Let \textsf{X} be \textsf{W} or \textsf{W*}. \il{X}-structure $\mathfrak{M}$ for $\mathcal{D}$ is a generalized Veltman model. Furthermore, the 
following holds:
$$ \mathfrak{M},w\Vdash G \ \mbox{ if and only if }\ G\in w,$$
for each $G\in\mathcal{D}$ and $w\in W.$ 
\label{lemma-main-w}
\end{lemma}
\begin{proof}
Let us first verify that the \il{X}-structure $\mathfrak{M} = (W, R, \{ S_w : w \in W \}, \Vdash )$ 
for $\mathcal{D}$ is a generalized Veltman model. 
All the properties, except for quasi-transitivity, have easy proofs (see the proof of Lemma 
\ref{lemma-main}).

Let us prove quasi-transitivity. Assume $u S_w V$, and $v S_w U_v$ for all $v \in V$. 
Put $U = \bigcup_{v \in V} U_v$. We claim that $u S_w U$. Clearly $U \subseteq R[w]$. 
To prove $uS_w U$ we will distinguish the cases (a) and (b) from the definition of the 
relation $S_w$ for $uS_w V.$

In the case (a), we have $v_0 \in V$ for some $v_0 \in \dot{R}[u]$. 
We will next distinguish two cases from the definition of $v_0 S_w U_{v_0}$. 

In the case (aa) we have $x \in U_{v_0}$ for some $x \in \dot{R}[v_0]$. 
Since $v_0 \in \dot{R}[u]$, we then have $x \in \dot{R}[u]$. 
Since $x \in U_{v_0} \subseteq U$, then $U\cap \dot{R}[u]\neq \emptyset.$
So, we have $uS_w U$, as required. 

In the case (ab) we have:
\[(\forall S)(w \prec_S v_0 \Rightarrow (\exists x \in U_{v_0})(\exists  G \in \mathcal{D} 
\cap \bigcup\dot{R}[v_0])\ w \prec_{S\cup \{\square\neg G\} } x).\]
To prove $uS_w U$ in this case, we will use the case (b) from the definiton of the relation $S_w$. 
Assume $w \prec_S u$. Then we have $w \prec_S u \prec v_0$ or $w \prec_S u = v_0$. 
Either way, possibly using Lemma Lemma \ref{lema3.2}, we have $w \prec_S v_0$, and so there are $x \in U_{v_0}$ and 
$G \in \mathcal{D} \cap \bigcup\dot{R}[v_0]$ with $w \prec_{S\cup \{\square\neg G\} } x$. 
Since $uRv_0$ or $u=v_0$, we have $\dot{R}[v_0] \subseteq \dot{R}[u].$ 
So, the claim follows.

\vskip 2ex
In the case (b), we have:
\[
      (\forall S)(w \prec_S u \Rightarrow (\exists v \in V)(\exists  G \in \mathcal{D} \cap 
      \bigcup\dot{R}[u]) \ w \prec_{S\cup \{\square\neg G\} } v).
		\]
To prove $uS_w U$ we will use the case (b) from the definition of the relation $S_w$. 
Assume $w \prec_S u$. Then there are $v_0 \in V$ and $G \in \mathcal{D} \cap \bigcup\dot{R}[u]$ such 
that $w \prec_{S\cup \{\square\neg G\} } v_0$. 
From $v_0 \in V$ it follows that $v_0 S_w U_{v_0}$. 
We will next distinguish the possible cases in the definition of $v_0 S_w U_{v_0}$. 

In the first case (ba) we have $U_{v_0}\cap \dot{R}[v_0]\neq \emptyset,$ i.e.\ there is 
$x \in U_{v_0} \cap \dot{R}[v_0]$. 
Then $w \prec_{S\cup \{\square\neg G\} } v_0 = x$ or $w \prec_{S\cup \{\square\neg G\} } v_0 \prec x$. 
In both cases (possibly using Lemma \ref{lema3.2}) we have $w \prec_{S\cup \{\square\neg G\} } x$. 

In the case (bb):
\[
         (\forall S')(w \prec_{S'} v_0 \Rightarrow (\exists x \in U_{v_0})(\exists  G' \in \mathcal{D} 
         \cap \bigcup\dot{R}[v_0])\ w \prec_{S'\cup \{\square\neg G'\} } x).  
\]
From $w \prec_{S\cup \{\square\neg G\} } v_0$ it follows that there are some 
$x \in U_{v_0}$ and $G' \in \mathcal{D} \cap \bigcup\dot{R}[v_0]$ such that 
$w \prec_{S\cup \{\square\neg G, \square\neg G'\} } x$. 
Lemma \ref{lema3.2} implies $w \prec_{S\cup \{\square\neg G\} } x$, as required.

\vskip 3ex
We claim that for each formula $G\in\mathcal{D}$ and each world $w\in W$ the following holds:
\[
    \mathfrak{M},w\Vdash G \ \mbox{ if and only if } \ G\in w.
\]
The claim is proved by induction on the complexity of $G$. 
The only non-trivial case is when $G=B\rhd C.$

Assume $B \rhd C \in w,$ \ $wR u$ and $u\Vdash B$. Induction hypothesis implies $B\in u.$
We claim that $u S_w [C]_w$. Clearly $[C]_w \subseteq R[w]$. 
Assume $w \prec_S u$. Lemma \ref{nedostaci-W} implies that there is an \il{X}-MCS $v$ with 
$w \prec_{S \cup \{ \square \neg B \} } v$ and $C, \square \neg C\in v$ (thus $v \in W$). 
Since $C \in v$, the induction hypothesis implies $v \Vdash C.$ 
Since $w \prec v$, i.e.\ $wRv$, then $v \in [C]_w$. 
Now, $B \in \mathcal{D}$ and $B\in u$ imply $B \in \mathcal{D} \cap \bigcup\dot{R}[u]$. 
Thus $uS_w [C]_w$ holds, by the clause (b) from the definition. 

To prove the converse, assume $B \rhd C \notin w$. 
Since $w$ is \il{X}-MCS,  $\neg(B\rhd C)\in w.$
Lemma \ref{problemi-W} implies there is $u$ with 
$w \prec_{\{ \square \neg B, \neg C\}} u$ and  $B \in u.$ 
Lemma \ref{lema3.2} implies $\square \neg B\in u.$ 
So, $B\wedge \square \neg B\in u$; thus $u \in W.$ 
The induction hypothesis implies $u \Vdash B$.
Let $V\subseteq R[w]$ be such that $u S_w V$.
We will find a world $v\in V$ such that $w\prec_{\{ \neg C\}} v$.  
We will distinguish the cases (a) and (b) from the definition of the relation $S_w$. 
Consider the case (a). 
Let $v$ be an arbitrary node in $V\cap \dot{R}[u]$. If $v = u$, clearly 
$w\prec_{\{\square\neg B,\neg C\}} v$.
If $uRv$, then we have $w\prec_{\{ \square\neg B,\neg C\}} u\prec v.$ 
Lemma \ref{lema3.2} implies $w\prec_{\{ \square\neg B,\neg C\}} v$. 
Consider the case (b). 
From $w\prec_{\{ \square\neg B,\neg C\}} u$ and the definition of $S_w$ it follows that there is $v \in V$ and a formula $D \in \mathcal D$ such that $w \prec_{\{\square \neg B, \neg C,\square\neg D\}} v$.
In both cases we have $w\prec_{\{ \neg C\}} v$; thus $C \notin v$.
Induction hypothesis implies $v \nVdash C$; whence $V \nVdash C$, as required.
\end{proof}

\begin{theorem}
The logic \il{W} is complete w.r.t.\ \ilgen{W}-models.
\label{tm-ilw}
\end{theorem}
\begin{proof}
In the light of Lemma \ref{lemma-main-w}, it suffices to show that the \il{W}-structure $\mathfrak{M}$ for $\mathcal{D}$ possesses the property 
\kgen{W}. Recall the characteristic property \kgen{W}: 
\[
  uS_w V \Rightarrow (\exists V' \subseteq V)(u S_w V' \ \& \ R[V'] \cap S_w^{-1}[V] = \emptyset ). 
\]


Suppose for a contradiction that there are $w$, $u$ and $V$ such that:
\begin{equation}
    \label{newgen}
     u S_w V \ \& \ (\forall V' \subseteq V)(u S_w V' \Rightarrow R[V'] \cap 
     S_w^{-1}[V] \neq \emptyset ).
\end{equation}
Let $\mathcal{V}$ denote all such sets $V$ (for arbitrary but fixed $w$ and $u$). 

Let $n = 2^{|\mathcal{D}|}$. Fix any enumeration $\mathcal{D}_0, \dots, \mathcal{D}_{n - 1}$  of $\mathcal{P}(\mathcal{D})$ that satisfies $\mathcal{D}_0=\emptyset.$
Denote a new relation $S_w^i$ for all $0 \leq i < n$, $y\in W$ and $U\subseteq W$ as follows:
   		\[
	  		y S_w^i U \iff yS_w U, \ \mathcal{D}_i \subseteq \bigcup\dot{R}[y], \ U \subseteq 
	  		\left[\bigvee_{G\in\mathcal{D}_i} \square \neg G \right]_w.
		\]

Let $y\in W$ and $U\subseteq W$ be arbitrary. Let us prove that $yS_w U$ implies the following:
		\begin{equation}
			\label{sv1}
			(\exists U' \subseteq U)(\exists i < n) \ y S_w^{i} U'.
		\end{equation}

If $yS_w U$ holds by (a) from the definition of $S_w$, the set $U\cap \dot{R}[y]$ is 
non-empty.
Pick arbitrary $z \in U\cap \dot{R}[y]$ and put $U' = \{z\}$. We have either $wRyRz$ or 
$y = z$. 
If $wRyRz$, we have $yS_w \{ z\}.$ 
Otherwise $y = z$. Now quasi-reflexivity implies $yS_w\{z\}.$
Since $y\in W$, there is a formula $G\in \mathcal{D}$ such that 
$G\wedge\square\neg G\in y$. 
Fix $i<n$ such that $\mathcal{D}_i=\{ G\}.$  
Clearly $\mathcal{D}_i\subseteq \bigcup\dot{R}[y].$
Since $z \in U$ and $yS_w U$, clearly $U'\subseteq R[w]$. 
Since $yRz$, we also have $\square\neg G\in z$. 
Truth lemma implies $U'\Vdash\square\neg G$; since if $zRt$, $G \notin t$, (truth lemma is applied here) $t \nVdash G$, so $z \Vdash \square \neg G$. Thus $U'\subseteq [\square\neg G]_w$, and $yS_w^i U'$.

If $yS_w U$ holds by (b) from the definition of $S_w$, take:
$$ U'=\{ z\in U : (\exists G\in \mathcal{D}\cap \bigcup\dot{R}[y]) \ 
w\prec_{\{ \square \neg G\} } z\};$$ 
$$ \mathcal{D}_i=\{ G \in \mathcal{D}\cap \bigcup \dot{R}[y] : (\exists z\in U) \
w\prec_{\{ \square\neg G\} } z\}.$$ 
In other words, $U'$ is the image of the mapping that is implicitly present in the definition of the relation $S_w$ (clause (b)): 
for each $S$, pick a world $v_S$ (to be included in $U'$), and a formula $G_S$ (to be included in $\mathcal{D}_i$). 


\vskip 2ex
Let $m < n$ be maximal such that there are $U \in \mathcal{V}$ and $U' \subseteq U$ with 
the following properties:
\begin{enumerate}
		\item[(i)] $(\forall x \in U)[ (\exists y \in R[x]) (\exists Z \subseteq U) 
		(\exists i \leq m)\ yS_w^{i} Z \Rightarrow x \notin U' ]$;
		\item[(ii)] $(\forall x \in W)(x S_w U \Rightarrow  xS_w U')$.
\end{enumerate}
		
Since $\mathcal{D}_0=\emptyset,$ we have
$[\bigvee_{G\in\mathcal{D}_0} \square\neg G]_w=[\bot]_w=\emptyset.$
So there are no $Z\subseteq [\bigvee_{G\in\mathcal{D}_0}\square\neg G]_w$ such that $yS_w Z$
for some $y\in W$. So, if we take $m = 0$ and $U' = U$ for any $U \in \mathcal{V}$, (i) and (ii) are trivially satisfied. 

Since $n$ is finite and conditions (i) and (ii) are satisfied for at least one value $m$, there must be a maximal $m < n$ with the 
required properties.

\vskip 2ex Let us first prove that $m < n - 1$. Assume the opposite, that is, $m = n-1$.
Then there are  $U\in\mathcal{V}$ and  $U'\subseteq U$ such that the conditions (i) and (ii) are satisfied for $m=n-1.$
Since $U\in\mathcal{V}$, we have $uS_w U.$
The condition (ii) implies $uS_w U'.$ 
Now $U\in\mathcal{V},$ \ $U'\subseteq U$ and $uS_w U'$ imply 
$R[U'] \cap S_w^{-1}[U] \neq \emptyset$.
Thus there are $x \in U'$ and  $y \in R[x]$ such that $y S_w U$. 
Now (ii) implies $y S_w U'$. 
The earlier remark (\ref{sv1}) implies that there is $Z\subseteq U'$ and $i < n$ such that 
$yS_w^i Z.$
Since $m=n-1$, it follows that $i\leq m.$
The condition (i) implies $x\not\in U',$ a contradiction.
Thus $m < n-1$.

\vskip 2ex
Let us now prove that $m$ is not maximal, by showing that $m + 1$ satisfies (i) and (ii).
Let $U\in\mathcal{V}$ and $U'\subseteq U$ be some sets such that the conditions (i) and 
(ii) are satisfied for $m$. 
Denote:
\[
		Y = \{ x \in U' : (\exists y \in R[x] )(\exists Z \subseteq U')\  
		yS_w^{m+1} Z\}.
\]
Let us prove that $m + 1$ also satisfies (i) and (ii) with $U'$ instead of $U$, and 
$U' \setminus Y$ instead of $U'$.
We should first show that $U' \in \mathcal V$. 
So, suppose that $u S_w T \subseteq U'$. 
Now, $T\subseteq U' \subseteq U$ and $U \in \mathcal V$ imply that there are some 
$v \in T$ and $z \in R[v]$ such that $zS_w U$.
The property (ii) for $m$ (with sets $U$ and $U'$) implies  $zS_w U'$. 
So, $R[T] \cap S_w^{-1}[U'] \neq \emptyset$, as required.

Now let us verify the property (i) for the newly defined sets. 
Let 
$x \in U', y \in R[x], Z \subseteq U', i \leq m + 1$ be arbitrary such that $yS_w^i Z$. 
If $i \leq m$, the property (i) for $m$ implies $x \notin U'$, so in particular, $x 
\notin U' \setminus Y$. 
If $i = m + 1$, then $x \in Y$. 
Thus $x \notin U' \setminus Y$ and the conditon (i) is satisfied.
				
It remains to prove (ii). 
Take arbitrary $x \in W$ such that $x S_w U'$.
For every $y \in  Y$, the definition of $Y$ implies existence of some $z_y \in R[y]$ and 
$U_y \subseteq U'$ such that $z_yS_w^{m+1} U_y$.
From the definiton of the relation $S_w^{m+1}$ we have 
$\mathcal{D}_{m+1}\subseteq \bigcup \dot{R}[z_y].$ 
Now, $yRz_y$ and the truth lemma imply $y\Vdash \Diamond G,$ for each 
$G\in \mathcal{D}_{m+1}.$ 
From the definiton of the relation $S_w^{m+1}$ and $z_yS_w^{m+1} U_y$ 
we have $U_y\subseteq [\bigvee_{G\in \mathcal{D}_{m+1}} \square\neg G]_w.$
So, the following holds:
\[Y \Vdash \bigwedge_{G \in \mathcal{D}_{m+1}} \Diamond G \quad \text{ and } \quad U_y 
\Vdash 
\bigvee_{ G \in \mathcal{D}_{m+1}}\square\neg G,\]
for all $y \in Y$.
Thus, $U_y \cap Y = \emptyset,$ for every $y\in Y.$ 
For every $y \in U' \setminus Y$ put $U_y = \{y\}$. 
Again, $U_y \cap Y = \emptyset$. 
Note that $\bigcup_{y \in U'} U_y = U' \setminus Y$.
Now $x S_w U'$ and quasi-transitivity imply $x S_w U' \setminus Y$.

The fact that (i) and (ii) hold for $m+1$ contradicts the maximality of $m$. 
\end{proof}

\vskip 2ex
Goris and Joosten proved in \cite{Goris-Joosten-08} the completeness of \il{W*} (\il{WM\textsubscript{0}}) w.r.t.\ ordinary  Veltman semantics. 

\begin{theorem}
\label{ilwst-potpunost}
The logic \il{W}* is complete w.r.t.\ \ilgen{W*}-models.
\end{theorem}
\begin{proof}
With Lemma \ref{lemma-main-w}, it suffices to prove the \il{W*}-structure for $\mathcal{D}$
possesses the properties \kgen{W} and \kgen{M\textsubscript{0}}, for each appropriate    $\mathcal{D}.$
So, let $\mathfrak{M}=(W,R,\{ S_w: w\in W\},\Vdash)$ be the \il{W}*-structure for $\mathcal D$.
Theorem \ref{tm-ilw} shows that the model $\mathfrak{M}$ possesses the property \kgen{W}.
It remains to show that it posesses the property \kgen{M\textsubscript{0}}.

\vskip 2ex
Assume $wRuRx S_w V$. We claim that there is $V'\subseteq V$ such 
that $u S_w V'$ and $R[V'] \subseteq R[u]$. 

First, consider the case when $xS_w V$ holds by the clause (a) from the 
definition of $S_w$.
So there is $v \in V$ such that $x=v$ or $xRv$.
In both cases, $wRuRv$, and so $uS_w \{ v\}.$ 
It is clear that $R[v]\subseteq R[x]\subseteq R[u]$. 
So it suffices to take $V' = \{ v \}$.

\vskip 2ex
Otherwise, $xS_w V$ holds by the clause (b). Take 
$V' = \{ v \in V : w \prec_{u_\emptyset^\square} v \}$. 
Clearly, $V'\subseteq V\subseteq R[w]$. 
Assume $w \prec_S u$. Now $w \prec_{S} u \prec x$ and Lemma \ref{label-mn} imply 
$w \prec_{S \cup u_\emptyset^\square} x$. 
The definition of $xS_wV$ (clause (b)) implies there is $G \in \mathcal{D} 
\cap \bigcup\dot{R}[x]$ (so $G \in \mathcal{D} \cap \bigcup\dot{R}[u]$) and $v \in V$ such that
$w \prec_{S \cup u_\emptyset^\square \cup \{ \square \neg G \}} v$, thus also $v \in V'$. 
In particular, $w \prec_{S \cup \{ \square \neg G \}} v$. Since $S$ was arbitrary, $u S_w V'$.
It remains to verify that $R[V'] \subseteq R[u]$. Assume $V' \ni vRz$. 
Since $w \prec_{u_\emptyset^\square} v$, for all $\square B \in u$ we have $\square B \in v$, and since 
$vRz$, it follows that $\square B, B \in z$. Thus, $u \prec z$ i.e.\  $uRz.$
\end{proof}

\vskip 2ex
In \cite{Mikec-Perkov-Vukovic-17} it is shown that \il{W*} possesses finite model property w.r.t.\ generalized Veltman models. To show decidability, (stronger) completeness w.r.t.\ ordinary Veltman models was used, but the Theorem \ref{ilwst-potpunost} would suffice for this purpose.

\section{Finite model property and decidability}
For \il{}, \il{M}, \il{P} and \il{W}, original completeness proofs were proofs of completeness w.r.t.\ appropriate finite models  \cite{deJongh-Veltman-90}, \cite{deJongh-Veltman-99}. For these logics, FMP w.r.t.\ ordinary semantics and decidability are immediate (and completeness and FMP w.r.t.\ generalized semantics are easily shown to follow from these results). These completeness proofs use \textit{truncated} maximal consistent sets, that is, sets that are maximal consistent with respect to the so-called \textit{adequate} set. The principal requirement is that this set is finite. Already with \il{M}, defining adequacy is not trivial (see \cite{deJongh-Veltman-90}). 

For more complex logics, not much is known about FMP w.r.t.\ ordinary semantics. The filtration method can be used with generalized models to obtain finite models. This approach was successfully used to prove FMP of \ilmn{} and \il{W*} w.r.t.\ generalized semantics \cite{Perkov-Vukovic-16}, \cite{Mikec-Perkov-Vukovic-17}. A drawback of this approach is in that FMP w.r.t.\ ordinary semantics does not follow from FMP w.r.t.\ generalized semantics. Decidability can be obtained from FMP w.r.t.\ either semantics. At the moment it is not clear whether the choice of semantics would affect our ability to produce results regarding computational complexity of provability and consistency of \il{X}. 

Let us overview basic notions and results of  \cite{Perkov-Vukovic-16} and \cite{Mikec-Perkov-Vukovic-17}. Let $A$ be a formula. If $A$ equals $\neg B$ for some $B$, then ${\sim}A$ is $B$, otherwise ${\sim}A$ is $\neg B$. We need to slightly extend the definition of adequate sets\footnote{Note that this is a different notion of \textit{adequacy} than the one used for completeness proofs in  \cite{deJongh-Veltman-90}, \cite{deJongh-Veltman-99}, and \cite{Goris-Joosten-08}.} that was used in \cite{Perkov-Vukovic-16}. The modified version will satisfy all the old properties.

\begin{definition}
Let $\mathcal{D}$ have the usual the properties: a finite set of formulas that is closed under taking subformulas and 
single negations, and $\top\in \mathcal{D}.$
We say that a set of formulas $\Gamma_\mathcal{D}$ is an adequate set w.r.t.\ 
$\mathcal{D}$ if it satisfies the following 
conditions:
\begin{enumerate}
\item $\Gamma_\mathcal{D}$ is closed under taking subformulas;
\item if $A\in\Gamma_\mathcal{D}$ then ${\sim}A\in\Gamma_\mathcal{D};$
\item $\bot\rhd\bot\in\Gamma_\mathcal{D};$
\item $A\rhd B\in\Gamma_\mathcal{D}$ if $A$ is an antecedent or succedent of some $\rhd$-formula 
in $\Gamma_\mathcal{D},$ and so is $B;$
\item if $A\in\mathcal{D}$ then $\square\neg A\in \Gamma_\mathcal{D}.$    
\end{enumerate}
\end{definition}

Since $\mathcal{D}$ is finite, $\Gamma_\mathcal{D}$ is finite too. Next we require the concept of bisimulations between generalized models.

\begin{definition}[\cite{Vrgoc-Vukovic}]
A bisimulation between generalized Veltman models 
$\mathfrak M=(W,R,\{S_w:w\in W\},\Vdash)$ 
and $\mathfrak M'=(W',R',\{S'_{w'}:w'\in W'\},\Vdash)$ is a non-empty
relation $Z\subseteq W\times W'$ such that:
\begin{itemize}
\item[(at)] if $wZw'$, then $w\Vdash p$ if and only if $w'\Vdash p$, for all propositional 
variables $p;$
\item[(forth)] if $wZw'$ and $wRu$, then there is $u'\in W'$ such that $w'R'u'$, $uZu'$ and for all 
$V'\subseteq W'$ such that $u'S'_{w'}V'$ there is $V\subseteq W$ such that $uS_wV$ and for
               all $v\in V$ there is $v'\in V'$ with $vZv';$ 
\item[(back)] if $wZw'$ and $w'R'u'$, then there is $u\in W$ such that $wRu$, $uZu'$ and for all 
$V\subseteq W$ such that $uS_wV$ there is $V'\subseteq W'$ such that $u'S'_{w'}V'$ and for
               all $v'\in V'$ there is $v\in V$ with $vZv'.$ 
\end{itemize}
\end{definition}

Given a generalized Veltman model $\mathfrak M$, the 
union of all bisimulations on $\mathfrak M$,
denoted by $\sim_\mathfrak{M}$, is the largest bisimulation on $\mathfrak M$, and $\sim_\mathfrak{M}$ is an equivalence relation \cite{Vrgoc-Vukovic}.

An $\sim_\mathfrak{M}$-equivalence class of $w\in W$ will be denoted by $[w]$. For any set of worlds $V$, put $\widetilde{V}=\{[w]:w\in V\}$.

A filtration of $\mathfrak M$ through $\Gamma_\mathcal{D},\ \sim_\mathfrak{M}$ is 
any generalized Veltman model 
$\widetilde{\mathfrak M}=(\widetilde{W},\widetilde{R},\{\widetilde{S}_{[w]}: w\in W\},\Vdash)$ such 
that for all $w\in W$ and $A\in\Gamma_\mathcal{D}$ we have $w\Vdash A$ if
and only if $[w]\Vdash A$ (we denote both forcing relations as $\Vdash$, as there is no risk of 
confusion).

The following lemma combines key results of \cite{Perkov-Vukovic-16} 
(Lemma 2.3, Theorem 2.4., Theorem 3.2).

\begin{lemma}\label{tm1}
Let $\mathfrak{M}=(W,R,\{S_w:w\in W\},\Vdash)$ be a generalized Veltman model, 
and $\sim_\mathfrak{M}$ the largest bisimulation on $\mathfrak{M}$. Define:

{\em (1)} $[w]\widetilde{R}[u]$ if and only if for some  $w'\in[w]$ and $u'\in[u]$, $w'Ru'$ and 
there is $\square A\in\Gamma_\mathcal{D}$ such that $w'\nVdash\square A$ and $u'\Vdash\square A;$

{\em (2)} $[u]\widetilde{S}_{[w]}\widetilde{V}$ if and only if $[w]\widetilde{R}[u]$, 
$\widetilde{V}\subseteq \widetilde{R}[[w]]$, and for all
$w'\in[w]$ and $u'\in[u]$ such that $w'Ru'$ we have $u'S_{w'}V'$ for some $V'$ such that 
$\widetilde{V'}\subseteq \widetilde{V};$

{\em (3)} for all propositional variables $p\in\Gamma_\mathcal{D}$
put $[w]\Vdash p$ if and only if 
$w\Vdash p,$ and interpret propositional variables $q\not\in\Gamma_\mathcal{D}$
arbitrarily (e.g.\ put $[w]\nVdash q$ for all $[w]\in\widetilde{W}).$

Then $\widetilde{\mathfrak{M}}=(\widetilde{W}, \widetilde{R}, \{ \widetilde{S}_{[w]} : w \in W\}, 
\Vdash)$ is a filtration of $\mathfrak{M}$ through $\Gamma_\mathcal{D},\sim_\mathfrak{M}$. The model $\widetilde{\widetilde{\mathfrak{M}}}$ is finite.
\end{lemma}

Lemma \ref{tm1} implies that \il{} has FMP w.r.t.\ generalized semantics. 
To prove that a specific extension has FMP, it remains to show that  filtration preserves its characteristic property.

Since we are going to use \il{X}-structures as the starting models $\mathfrak{M}$, we can make use of their properties. In particular, we do not have to make sure that there is a formula $\square A$ such that $x \nVdash \square A$ and $y \Vdash \square A$ when we want to show that $xRy$ implies $[x] \R [y].$

\begin{lemma}\label{lema-pomoc}
Let $X$ be a subset of $\{$\textsf{M}, \textsf{M\textsubscript{0}}, \textsf{P}, \textsf{P\textsubscript{0}}, \textsf{R}, \textsf{W}$\}$ 
and $\mathfrak M = (W, R, \{S_w : w \in W\}, \Vdash)$ 
the \il{X}-structure for a set $\mathcal{D}.$ 
If $xRy$ then $[x] \R [y]$. 
\end{lemma}
\begin{proof}
We should find a formula $\square A \in \Gamma_\mathcal{D}$ 
such that $x \nVdash \square A$ and $y \Vdash \square A$. 
Since $y \in W$, there is a formula $B \in \mathcal D$ such that 
$B\wedge\square\neg B\in y$. Lemma \ref{lemma-main} implies
$y \Vdash B \wedge \square \neg B$ ($\square \neg B$ might not be in $\mathcal{D}$, but ${\sim} B$ is; since $R[y] \Vdash {\sim} B$, we must have $y \Vdash \square \neg B$). 
Since $xRy$, we have $x \nVdash \square \neg B$. Since 
$B \in \mathcal D$, we have $\square \neg B \in \Gamma_\mathcal{D}$.
Thus we can take $A = \neg B$.
\end{proof} 

Note that the proof above applies not only when transforming \il{X} structures $\mathfrak M$ to $\widetilde{\mathfrak{M}}$, but also when transforming $\widetilde{\mathfrak{M}}$ to $\widetilde{\widetilde{\mathfrak{M}}}$.

\begin{lemma}\label{ILP0-fmp}
Let $\mathfrak{M}=(W,R,\{ S_w : w\in W\},\Vdash)$ be the \il{P\textsubscript{0}}-structure for 
$\mathcal{D}$  
and  $\sim_\mathfrak{M}$ be the largest bisimulation on $\mathfrak{M}.$
Then the filtration $\widetilde{\mathfrak M}$ as defined in Lemma \ref{tm1} possesses the 
property \kgen{P\textsubscript{0}}.
\end{lemma}
\begin{proof}
Assume $[w] \widetilde{R} [x] \widetilde{R} [u] \widetilde{S}_{[w]} V$ and
$\widetilde{R}[[v]]\cap Z\neq \emptyset $ for each $[v]\in V.$
We claim that there exists $Z'\subseteq Z$ such that $[u]\widetilde{S}_{[x]} Z'.$

Since $[w] \R [x]$, there are $w_0 \in [w]$ and $x_0 \in [x]$ such that $w_0Rx_0$. 
Let $x'\in[x]$ and $u'\in[u]$ be any worlds such that $x'Ru'$. 
The condition (back) implies that there is a world $u_{x', u'}$ such that $x_0Ru_{x', u'}$ and 
$u_{x', u'} \sim_{\mathfrak{M}} u'$.
Now, $[u]\widetilde{S}_{[w]} V$, $u_{x', u'}\in[u]$ and $w_0 R u_{x', u'}$ imply
there is a set $V_{x', u'}$ such that $u_{x', u'} S_{w_0} V_{x', u'}$ and 
$\widetilde{V_{x', u'}}\subseteq V.$
Since $\widetilde{R}[[v]]\cap Z\neq\emptyset$ for each $[v]\in V$,  we have
$\widetilde{R}[[v]]\cap Z\neq \emptyset$ for each $v\in V_{x', u'}$.
For each $v\in V_{x', u'}$, choose a world $z_v$ such that $[z_v]\in \widetilde{R}[[v]]\cap Z$.
Now $[v]\widetilde{R}[z_v]$ implies that there are some $v'\in [v]$ and $z'_v\in [z_v]$ such that $v'Rz_v'$.
Applying (back), we can find a world $z''_v$ such that $vRz''_v$ and 
$z'_v\sim_\mathfrak{M} z''_v$.
Put $Z_{x', u'}=\{ z''_v : v\in V_{x', u'} \}$.
Note that we have $R[v]\cap Z_{x', u'}\neq \emptyset$ for each $v\in V_{x', u'}.$

Applying \kgen{P\textsubscript{0}} gives $u_{x', u'} S_{x_0} Z'_{x', u'}$ for some 
$Z'_{x', u'}\subseteq Z_{x', u'}$. 
Clearly $\widetilde{Z'_{x', u'}}\subseteq \widetilde{Z_{x', u'}}\subseteq Z$.
Continuing our first application of (back), there is a set $Z''_{x', u'}$ such that $u'S_{x'} Z''_{x', u'}$, and for each 
$z''\in Z''_{x', u'}$ there is $z'\in Z'_{x', u'}$ such that $z' \sim_\mathfrak{M} z''$.
This implies $\widetilde{Z''_{x', u'}}\subseteq \widetilde{Z'_{x', u'}}$.
Let $T = \bigcup_{x' \in [x], u' \in [u], x'Ru' } Z''_{x', u'}$ and $Z' = \widetilde{ T }$. 
It is easy to see that $Z'\subseteq Z$. 
Lemma \ref{lema-pomoc} implies $Z'\subseteq \widetilde{R}[[x]]$. 
We have $u'S_{x'} Z''_{x', u'}$ with $\widetilde{Z''_{x', u'}} \subseteq Z'$ for all $x' \in [x]$ and $u' \in [u]$ with 
$x'Ru'$.
Thus, $[u]\widetilde{S}_{[x]} Z'.$
\end{proof}

\begin{corollary}
\il{P\textsubscript{0}} is decidable.
\end{corollary}

Let us prove the same for \il{R}.

\begin{lemma}\label{ILR_fmp}
Let $\mathfrak{M}=(W,R,\{ S_w : w\in W\},\Vdash)$ be the \il{R}-structure for $\mathcal{D}$  and
 $\sim_\mathfrak{M}$ be the largest bisimulation on $\mathfrak{M}.$
Then the filtration $\widetilde{\mathfrak M}$ as defined in Lemma \ref{tm1} possesses the 
property \kgen{R}.
\end{lemma}		
\begin{proof}		
Assume $[w] \Rt [x] \Rt [u] \Swt {V}$, and let $C \in \mathcal{C}([x], [u])$ be an arbitrary choice set.
We are to prove that there is a set $U$ such that $\widetilde{U}\subseteq V,$ $[x]\widetilde{S}_{[w]} \widetilde{U}$ and
$\widetilde{R}[\widetilde{U}]\subseteq C$.

Put $C_{x'} = \{ c \in R[x'] : [c] \in C \}$ for all $x'\in [x]$.

Let us first prove that for some $x_0\in [x], u_0\in [u]$ with $x_0Ru_0$ we have 
$C_{x_0} \in \mathcal{C}(x_0, u_0)$. 
Suppose not. Then for all $x'\in [x], u'\in [u]$ 
with $x'Ru'$, there is a set $Z_{x', u'}$ such that 
$u' S_{x'} Z_{x', u'}$ with $Z_{x',u'} \cap C_{x'} = \emptyset$. 
Put $Z = \bigcup_{x'\in [x], u'\in [u], x'Ru'} Z_{x', u'}$.
Lemma \ref{lema-pomoc} implies $\widetilde{Z}\subseteq \widetilde{R}[[x]]$. 
Thus $[u]\St{x} \widetilde{Z}$.  
Since $C\in \mathcal{C}([x],[u])$, there is $z\in Z$ such that $[z]\in C\cap \widetilde{Z}$. 
Thus $z \in Z_{x', u'}$ for some $x'\in [x], u'\in [u]$ and $x'Ru'$. The definition of $C_{x'}$ 
implies $z \in C_{x'}$. Thus, $Z_{x',u'} \cap C_{x'} \neq \emptyset$, a contradiction. 

Now we claim that for all $y \in [x]$ there is $u_y \sim u_0$ with $y R u_y$ and 
$C_{y} \in \mathcal{C}(y, u_y)$. 
Since $y \sim_\mathfrak{M} x_0$ and $x_0Ru_0$,  the (back) condition implies that there is a world $u_y$ such that $u_y \sim_\mathfrak{M} u_0$ and $yRu_y$ (and other properties that we will return to later). 
We will show that $C_{y} \in \mathcal{C}(y, u_y)$.
Let  $Z'$ be such that $u_y S_y Z'$, and we are to prove that 
$C_{y} \cap Z' \neq \emptyset$. 
The earlier instance of (back) condition for $u_y$ further implies that there is a set $Z$ with $u_0 S_{x_0} Z$, and for all 
$z \in Z$ there is $z' \in Z'$  with  $z \sim_\mathfrak{M} z'$. 
Let $z\in Z \cap C_{x_0}$ be an arbitrary element (which exists because, as we proved, $C_{x_0}$ is a choice set). 
So there is $z' \in Z'$ such that $z' \sim_\mathfrak{M} z$.  
Since $[z] \in C$, equivalently $[z'] \in C$, we have $z'\in C_{y}$. In particular, $Z' \cap C_{y} \neq \emptyset$. Thus $C_y \in \mathcal{C}(y, u_y)$. 

Let us prove that there is a set $U$ such that $\widetilde{U}\subseteq V,$ $[x]\widetilde{S}_{[w]} \widetilde{U}$
and $\widetilde{R}[\widetilde{U}]\subseteq C.$
Let $w' \in [w]$ and $y \in [x]$ be such that $w'Ry$. 
Since $[u]\Swt V$, there is a set 
$V_{w',y}$ such that $u_y S_{w'} V_{w',y}$ and $\widetilde{V_{w',y}} \subseteq V$.
Applying $\kgen{R}$ with $C_y$, there is $U_{w',y}\subseteq V_{w',y}$ such that $y S_{w'} U_{w', y}$ and  
$R[U_{w',y}] \subseteq C_{y}$. 
Let $U= \bigcup_{w' \in [w], y \in [x], w'Ry} U_{w',y}$. 
Clearly $\widetilde{U}\subseteq V.$
Lemma \ref{lema-pomoc} implies $\widetilde{U}\subseteq \widetilde{R}[[w]].$
The definition of $\Swt$ implies $[x] \Swt \widetilde{U}$. 

It remains to verify that $\Rt[\widetilde{U}] \subseteq C$. 
Let $t\in U $ and $z\in W$ be such that $[t]\widetilde{R} [z]$.
Then we have $t \in U_{w', y}$ for some $w' \in [w]$ and $y \in [x]$. 
Since $[t]\Rt [z]$, there are $t'\in [t]$ and $z' \in [z]$ with $t' R z'$. 
The condition (forth) implies that there is $z''$ such that 
$tRz''$ and $z'\sim_\mathfrak{M} z''.$ 
Since $R[U_{w',y}]\subseteq C_{y})$ and $z''\in R[U_{w',y}]$, we have $z''\in C_{y}$.
Definition of $C_{y}$ implies $[z'']\in C$, or equivalently, $[z] \in C$. 
\end{proof}	

\begin{corollary}
\il{R} is decidable.
\end{corollary}


\end{document}